\documentclass{article}
\usepackage{amssymb,amsthm,amsmath} 
\usepackage{a4wide}
\usepackage[T1]{fontenc}
\usepackage[utf8]{inputenc}         

\usepackage{float}
\newcommand{\F}{\mathbb{F}}
\usepackage{graphicx} 
\usepackage{a4wide}
\usepackage{enumerate}
\usepackage{enumitem, hyperref}
\usepackage{tikz}
\tikzstyle{main node}=[circle, draw,
                        inner sep=0pt, minimum width=15pt]
\tikzstyle{code node}=[circle, draw, fill=lightgray,
                        inner sep=0pt, minimum width=15pt]

\theoremstyle{plain}
\newtheorem{theorem}{Theorem}

\newtheorem{lemma}[theorem]{Lemma}
\newtheorem{corollary}[theorem]{Corollary}
\theoremstyle{definition}

\newcommand{\Z}{\mathbb{Z}}

\newcommand{\bc}{\mathbf{c}}

\newcommand{\ba}{\mathbf{a}}
\newcommand{\bb}{\mathbf{b}}
\newcommand{\bv}{\mathbf{v}}
\newcommand{\bw}{\mathbf{w}}
\newcommand{\bx}{\mathbf{x}}

\newcommand{\bu}{\mathbf{u}}
\newcommand{\nolla}{\mathbf{0}}
\newcommand{\e}{\mathbf{e}}

\makeatletter
\def\namedlabel#1#2{\begingroup
    #2%
    \def\@currentlabel{#2}%
    \phantomsection\label{#1}\endgroup
}
\makeatother

\newtheorem{innercustomthm}{Theorem}
\newenvironment{customthm}[1]
  {\renewcommand\theinnercustomthm{#1}\innercustomthm}
  {\endinnercustomthm}

\begin{document}

\title{Improved Lower Bound for Locating-Dominating Codes in Binary Hamming Spaces}

\author{\textbf{Ville Junnila}, \textbf{Tero Laihonen} and \textbf{Tuomo Lehtil{\"a}}\thanks{Research supported by the Jenny and Antti Wihuri foundation.}\\
Department of Mathematics and Statistics\\
University of Turku, FI-20014 Turku, Finland\\
viljun@utu.fi, terolai@utu.fi and tuomoleh@gmail.com} 
\date{}
\maketitle
\noindent

\begin{abstract}
In this article, we study locating-dominating codes in binary Hamming spaces $\F^n$. Locating-dominating codes have been widely studied since their introduction in 1980s by Slater and Rall. They are dominating sets suitable for distinguishing vertices in graphs. Dominating sets as well as locating-dominating codes have been studied in Hamming spaces in multiple articles. Previously, Honkala et al. (2004) have presented a lower bound for locating-dominating codes in binary Hamming spaces. In this article, we improve the lower bound for all values $n\geq10$. In particular, when $n=11$, we manage to improve the previous lower bound from $309$ to $317$. This value is very close to the current best known upper bound of $320$.
\end{abstract}\textbf{Keywords:}  Locating-dominating set; locating-dominating code; Hamming space; share technique

\section{Introduction}

In this paper, we consider a graph $\F^n$, which is the \textit{binary Hamming space of dimension} $n$, that is, the binary $n$-cube. The graph has $2^n$ vertices and its vertex set consists of binary \textit{words} of length $n$ consisting of zeroes and ones. Now two words have an edge between them if they differ in exactly one position. For example, $110$ and $100$ are adjacent in $\F^3$. Let $\bu=x_1x_2\dots x_n\in \F^n$ be a word where each $x_i\in\{0,1\}$. We denote by $\e_i=x_1\dots x_n$ the word where $x_i=1$ and $x_j=0$ for each $j\neq i$. The set of integers $\{i\mid a\leq i\leq b\}$ for some $a,b\in\Z$ will be denoted by $[a,b]$. Let us denote by $$w(\bu)=|\{x_i=1\mid \bu=x_1x_2\dots x_n, i\in[1,n]\}|$$ the \textit{weight} of the word $\bu$. The \textit{distance} $d(\bu,\bv)$ of two words $\bu$ and $\bv$ is equal to the number coordinates in which they differ. Thus,  $$d(\bu,\bv)=w(\bu+\bv)=|\{x_i\neq y_i\mid\bu=x_1x_2\dots x_n\text{ and } \bv=y_1y_2\dots y_n, 1\leq i\leq n\}|$$ where $\bu+\bv$ is a coordinatewise sum modulo $2$. Moreover, we denote by $N(\bu)$ the \textit{open neighbourhood} of the word $\bu$, that is, the set of words adjacent to the word $\bu$ (i.e., the set of words at distance one to word $\bu$). Moreover, by $N[\bu]$, we denote the set $N(\bu)\cup \{\bu\}$ called the \textit{closed neighbourhood} of $\bu$. We are also interested in sets of words which are farther away from a certain word. Hence, we denote \textit{a ball of radius $r$ centered at the word} $\bu$ by $$B_r(\bu)=\{\bv \in \F^n \mid d(\bu,\bv)\leq r\}.$$ Notice that $N[\bu]=B_1(\bu)$.

\textit{A code} $C$ in $\F^n$ is a nonempty set of words in the Hamming space. We denote by $$I(C;\bu)=I(\bu)=C\cap N[\bu]$$ the $I$\emph{-set} of the word $\bu$. We say that a code $C$ is \textit{dominating} if we have $I(C;\bu)\neq \emptyset$ for each word $\bu\in\F^n$. For an introduction on domination in graphs, see \cite{haynes1998fundamentals}. Moreover, a code $C$ is said to be \textit{locating-dominating} if it is dominating and for each pair of distinct non-codewords $\bu,\bv\not\in \F^n \setminus C$ we have $$I(\bu)\neq I(\bv).$$ We also denote $$I_r(C;\bu)= I_r(\bu) = C\cap B_r(\bu)$$ for $r\geq1$. In this article, we are mostly interested in the smallest possible locating-dominating codes in the Hamming space $\F^n$. The cardinality of such a code is denoted by $\gamma^{LD}(\F^n)$ and any code attaining this cardinality is called \textit{optimal}.

Slater and Rall have originally introduced locating-dominating codes in $1980$s, see for example, \cite{rall1984location, slater1987domination, slater1988dominating}. One of the motivations to study locating-dominating codes has been \textit{sensor networks}. In this context, we place sensors on some set of vertices determined by the locating-dominating code. Moreover, we assume that a sensor sends alarm $1$ if there is an intruder/object/malfunction in any neighbouring vertex and alarm $2$ if the problem is in the same vertex as the sensor itself. Now, since the sensor placement is done using the locating-dominating code, we can deduce the location of the object just by considering which sensors are sending the alarm. The topic of locating-dominating codes has attracted a lot of attention recently, see \cite{foucaud2020domination, hernando2019locating, hudry2019unique, junnila2019stronger, junnila2019optimal}. For more papers in the field consult, the bibliography \cite{lobstein2012watching}.

In \cite{slater2002fault}, Slater has given the following general lower bound for locating-dominating codes in Hamming spaces. \begin{theorem}[\cite{slater2002fault}, Theorem 2]
We have $$\gamma^{LD}(\F^n)\geq \frac{2^{n+1}}{n+3}$$
\end{theorem}
\noindent
In $2004$, Honkala \textit{et al.} \cite{honkala2004locating}, gave the following improvement on Slater's result.

\begin{theorem}[\cite{honkala2004locating}, Theorem $15$]\label{VanhaRaja}
We have
$$\gamma^{LD}(\F^n)\geq \frac{n^22^{n+1}}{n^3+2n^2+3n-2}.$$
\end{theorem}

After the paper \cite{honkala2004locating}, there has been no advances in the lower bound. The goal of this article is to improve the bound of Theorem \ref{VanhaRaja} when $n\geq10$. The new bounds are presented in Corollary \ref{perusCorollary} and Theorem \ref{parasraja}. The new lower bound is especially interesting for $n=11$. In that case, the earlier bound of Theorem \ref{VanhaRaja}, gave the lower bound $309$. In Theorem \ref{parasraja}, we improve this bound to $\gamma^{LD}(\F^{11})\geq317$. Moreover, in \cite[Corollary 27]{junnila2018regular}, there is a construction of cardinality $320$ for locating-dominating codes in $\F^{11}$. Thus, the possible cardinality of an optimal locating-dominating code in $\F^{11}$ is in $[317,320]$ and the gap between the lower and the upper bound is rather small. In fact, even in the binary Hamming spaces $\F^n$ where $n\in[7,10]$ the gap between the lower and upper bound is larger as we can see in Table \ref{Table}.

Observe that many of the upper bounds (those with Key (c)) in Table \ref{Table} are based on \textit{identifying codes}. A code $C\in\F^n$ is said to be identifying \cite{karpovsky1998new} if it is dominating and for each pair of distinct words $\bv,\bu \in \F^n$  we have $$I(\bv)\neq I(\bu).$$ In particular, every identifying code is also a locating-dominating code. An interested reader may find more information about identifying codes, for example, in \cite{auger2010minimal, charon2010new,  hudry2019unique, junnila2019conjecture}.

%
\begin{table}[]\center{
\begin{tabular}{|l|l|l||l|l|l|l|}
	\hline
	$n$ & Lower bound & Upper bound & $n$  & Lower bound & Lower bound (New) & Upper bound \\ \hline\hline
	$1$ & $1$ (A)         & $1$ (A)     & $8$  & $50$ (B)       &             & $61$ (e)    \\ \hline
	$2$ & $2$ (A)         & $2$ (A)     & $9$  & $91$ (B)       &             & $112$ (c)   \\ \hline
	$3$ & $4$ (A)         & $4$ (A)     & $10$ & $167$ (B)      & $171$       & $208$ (c)   \\ \hline
	$4$ & $6$ (A)         & $6$ (A)     & $11$ & $309$ (B)      & $317$       & $320$ (d)   \\ \hline
	$5$ & $10$ (A)        & $10$ (A)    & $12$ & $576$ (B)      & $589$       & $640$ (d)   \\ \hline
	$6$ & $16$ (B)        & $18$ (A)    & $13$ & $1077$ (B)     & $1099$      & $1280$ (c)  \\ \hline
	$7$ & $28$ (B)        & $32$ (A)    & $14$ & $2023$ (B)     & $2059$      & $2550$ (c)  \\ \hline
\end{tabular}}\caption{Bounds for locating-dominating codes in small Hamming spaces. The second and fifth columns contain old lower bounds and the sixth column new lower bounds. The keys of the table ar as follows: (A) \cite{honkala2004locating}, (B) Theorem \ref{VanhaRaja}, (c) \cite{charon2010new}, (d) \cite{junnila2018regular} and (e) \cite[Appendix]{ranto2007identifying}.}\label{Table}
\end{table}
The best known bounds for the locating-dominating codes in $\F^n$ have been presented in Table \ref{Table}. The keys of the table are presented below:
\begin{enumerate}
\item[(A)] \cite{honkala2004locating} contains some trivial bounds as well as computer searches and a theorem for systematic upper bounds for locating-dominating codes.
\item[(B)] Lower bound presented in Theorem \ref{VanhaRaja}.
\item[(c)] \cite{charon2010new} contains constructions for identifying codes.
\item[(d)] \cite{junnila2018regular} contains some constructions for locating-dominating codes.
\item[(e)] \cite[Appendix]{ranto2007identifying} contains a computer search for a locating-dominating code in $\F^n$. 
\end{enumerate}

\medskip

In this paper, we combine the approaches of \cite{honkala2004locating} and the share technique considered in \cite{slater2002fault}. We apply the share technique by introducing a set of three rules for averaging the share among the codewords combined with some careful analysis of the structure of the binary Hamming space.

The following result is well-known and easily verifiable.
\begin{lemma}\label{palloleikkaus}
Let $\ba,\bb\in \F^n$. We have
$$\big|N[\ba]\cap N[\bb]\big|=\begin{cases}
 0, & \text{if }d(\ba,\bb)\geq 3 
\cr 2, & \text{if }d(\ba,\bb)=2
\cr 2, & \text{if }d(\ba,\bb)=1
\cr n+1, & \text{if }\ba=\bb.
\end{cases}$$
\end{lemma}
Notice that if we have $\ba,\bb\in \F^n$ and $d(\ba,\bb)=1$, then $N[\ba]\cap N[\bb]=\{\ba,\bb\}$. Moreover, if $\ba=\bb$, then $N[\ba]\cap N[\bb]=N[\ba]=\{\ba\}\cup\{\ba+\e_i\mid i\in[1,n]\}$. Furthermore, if $d(\ba,\bb)=2$, then $\ba=\bb+\e_i+\e_j$, $i\neq j$, and $N[\ba]\cap N[\bb]=\{\bb+\e_i,\bb+\e_j\}$.

\section{The first bound}\label{first bound}


Let $C\subseteq \F^n$ be a locating-dominating code. Since $C$ is a dominating code, we have $|I(\bv)|\geq1$ for all $\bv\in \F^n$. Now there can exist only four different types of words in $\F^n$ as we will explain below.

\begin{enumerate}
\item[\namedlabel{i}{$(i)$}] A word $\bv$ such that $|I(\bv)|=1$.
\item[\namedlabel{ii}{$(ii)$}] Codeword pairs or \textit{couples} $\bc_1,\bc_2$ such that $I(\bc_1)=I(\bc_2)=\{\bc_1,\bc_2\}$.
\item[\namedlabel{iii}{$(iii)$}] A word $\bu$ such that $|I(\bu)|\geq3$.
\item[\namedlabel{iv}{$(iv)$}] A word $\bx$ such that $|I(\bx)|=2$ and $I(\bx)\subset I(\bu)$ for some $\bu$ such that $|I(\bu)|\geq3$.
\end{enumerate}

We call the words of Type \ref{iv}  \textit{sons} and words of Type \ref{iii} \textit{fathers}. If $\bx$ is a son and $I(\bx)\subseteq I(\bu)$, then the word $\bu$ is the father of $\bx$. Notice that the son $\bx$ has only one father by Lemma \ref{palloleikkaus}. The set of a father and its sons forms a \textit{family}. Observe that it is possible to have no sons in a family, that is, a family consisting of a single father. Notice that together Cases \ref{ii} and \ref{iv} consider every possible word with $I$-set of cardinality two. Indeed, if we have a non-codeword $\bx$ with $I(\bx)=\{\bc,\bc'\}$, then $d(\bc,\bc')=2$ and, by Lemma \ref{palloleikkaus}, we have $|N[\bc]\cap N[\bc']|=2$. Hence, there exists another word $\bu$ such that $I(\bx)=\{\bc,\bc'\}\subseteq I(\bu)$. Since $C$ is a locating-dominating code, we have $|I(\bu)|\geq3$ and therefore, $\bu$ is the father of the son $\bx$. Furthermore, if we have a codeword $\bc$ such that $I(\bc)=\{\bc,\bc'\}$, then $I(\bc)\subseteq I(\bc')$. If $|I(\bc')|=2$, then $\bc$ and $\bc'$ form a couple. If $|I(\bc')|\geq3$, then $\bc$ is the son of the father $\bc'$. 

Furthermore, let us say that a father which is covered by $i\geq3$ codewords is an $F_i$\emph{-father} and by $F_i(\bc)$ we denote the number of $F_i$-fathers in $N[\bc]$.  
 Moreover, if a word of Type \ref{i} is a non-codeword, then we call it an \textit{orphan}.

The \textit{share} of a codeword $\bc\in C$ is introduced by Slater in \cite{slater2002fault} as $$s(\bc)=\sum_{\bv\in N[\bc]}\frac{1}{|I(\bv)|}.$$ Moreover, we have
\begin{equation}\label{shareIdea}
\sum_{\bc\in C}s(\bc)=|\F^n|=2^n.
\end{equation}
 Thus, if $s(\bc)\leq \alpha$ for each $\bc\in C$ and some positive real constant $\alpha$, then we have $|C|\alpha\geq2^n$ and hence, \begin{equation}\label{shareIdea2}|C|\geq \frac{2^n}{\alpha}.\end{equation}

Notice that we may calculate the share of a codeword $\bc$ if we know, in its neighbourhood, the number words of Types $(i)$ to $(iv)$ and $F_i(\bc)$ for each $3\leq i\leq n+1$. In particular, each $F_i$-father in $N[\bc]$ contributes $1/i$ to $s(\bc)$, each son $1/2$, each orphan $1$ and if $\bc$ is in a couple, then the couple contributes total of $1$ to the share. Therefore, we have $$s(\bc)=\left(\sum_{i=3}^{n+1} \sum_{\substack{\bv\in N[\bc] \\ \bv \text{ is an } F_i\text{-father}}}\frac{1}{i}\right)+\sum_{\substack{\bx\in N[\bc] \\ \bx\text{ is a son}}}\frac{1}{2}+\sum_{\substack{\bc\text{ and }\bc'\\\text{are a couple}}}1+\sum_{\substack{\bu\in N[\bc]\\ \bu\text{ is an orphan}}}1+\sum_{I(\bc)=\{\bc\}}1.$$
Obviously every sum cannot be non-empty simultaneously. For example, if $\bc$ and $\bc'$ form a couple, then $I(\bc)\neq\{\bc\}$. Observe that $|N[\bc]|=n+1$ and there is at most one orphan in $N[\bc]$ for $\bc\in C$ if $C$ is a locating-dominating code. Indeed, if we have two orphans $\ba,\bb\in N[\bc]$, then we have $I(\ba)=I(\bb)=\{\bc\}$.


Let us then consider the number of sons and their fathers in the vicinity of some codeword $\bc_1$. Let $\bv\in N[\bc_1]$ be a son with $I$-set $I(\bv)=\{\bc_1,\bc_2\}$. Since $\bv$ is a son, there exists a father $\bu\neq \bv$, such that $I(\bv)\subseteq I(\bu)$ and $I(\bu)=\{\bc_1,\bc_2,\dots, \bc_h\}$ where $h\geq3$ by \ref{iv} (notice that $\bv$ might be $\bc_1$ or $\bc_2$ in which we have $\{\bu,\bv\}=\{\bc_1,\bc_2\}$ and hence, $\bc_1$ does not form a couple with $\bc_2$).  Moreover, there can be at most $h-1$ sons in $N[\bc_1]$ with $\bu$ as their father. Namely, those with $I$-set equal to $\{\bc_1,\bc_i\}$ where $i\in[2,h]$. Hence, for each son in $N[\bc_1]$, we also have a father in $N[\bc_1]$ and for each $F_h$-father in $N[\bc_1]$, we have at most $h-1$ sons in $N[\bc_1]$. Therefore, if we have $t$ sons in $N[\bc]$, then  we have \begin{equation}\label{isaMaara}
\sum_{i=3}^{n+1}(i-1)F_i(\bc)\geq t.
\end{equation}


In general, if we have $F_{i_1}$- and $F_{i_2}$-fathers, $i_1\geq i_2\geq4$, in $N[\bc]$, $\bc\in C$, then the sum of the individual contributions of the fathers to the share $s(\bc)$ can be estimated as follows


\begin{equation}\label{Apuyhtälö}
\frac{1}{i_1}+\frac{1}{i_2}< \frac{1}{i_1+1}+\frac{1}{i_2-1}\text{, when } i_1\geq i_2\geq2.
\end{equation}

\medskip

In what follows, we consider the following types of codewords:
\begin{enumerate}
\item[\namedlabel{1}{$1)$}] A codeword without orphans in their neighbourhood.
\item[\namedlabel{2}{$2)$}] A codeword $\bc$ such that $|I(\bc)|\geq2$.
\item[\namedlabel{3}{$3)$}] A codeword $\bc$ with at most $n-4$ sons, one orphan in $N(\bc)$ and $I(\bc)=\{\bc\}$.
\item[\namedlabel{4}{$4)$}] A codeword $\bc$ with $n-3$ sons, one orphan in $N(\bc)$ and $I(\bc)=\{\bc\}$.
\item[\namedlabel{5}{$5)$}] A codeword $\bc$ with $n-2$ sons, one orphan in $N(\bc)$ and $I(\bc)=\{\bc\}$.
\item[\namedlabel{6}{$6)$}] A codeword $\bc$ with more that $n-2$ sons, one orphan in $N(\bc)$ and $I(\bc)=\{\bc\}$.
\end{enumerate}


\begin{lemma}\label{TyyppiLemma}
Let $n\geq10$ and $C\subseteq\F^n$ be a locating-dominating code. Then the following properties hold: 
\begin{enumerate}
\item We have $s(\bc)\leq n/2+1+1/(n-1)$ for each $\bc\in C$.
\item If $\bc\in C$ is not of Type \ref{5}, then $s(\bc)\leq n/2+1$.
\item If $n\geq13$ and $\bc\in C$ is not of Type \ref{5}, then $s(\bc)\leq \frac{n}{2}+1+\frac{1}{n^2-5n}-\frac{2}{3n}$.
\end{enumerate}
\end{lemma}
\begin{proof}
Next we will bound the share of a codeword from above. Roughly saying, the less fathers there are in the neighbourhood $N[\bc]$, the greater the share $s(\bc)$ is. Observe that for  $n\geq13$ \begin{equation}\label{ApuEYht}
\frac{n}{2}+1-\frac{1}{n-1}\leq\frac{n}{2}+1+\frac{1}{n^2-5n}-\frac{2}{3n}.
\end{equation} This inequality will be useful multiple times in this proof.  

Let us first consider a codeword $\bc\in C$ of Type \ref{1}. Moreover, we may assume that there is at least one father in $N[\bc]$. Indeed, if we have even one son in $N[\bc]$, then by definition we have at least one father in $N[\bc]$. Moreover, we have at least one son or father in $N[\bc]$ since $\bc$ belongs to at most one couple and 
there are no orphans in $N[\bc]$ since $\bc$ is of Type \ref{1}. Furthermore, $\bc$ is the only word in $N[\bc]$ which may contribute $1$ to the share, a father contributes at most $1/3$ and all other words contribute at most $1/2$ (including the codewords in a couple). Therefore, we have $$s(\bc)\leq 1+\frac{1}{3}+\frac{n-1}{2}=\frac{n}{2}+\frac{5}{6}< \frac{n}{2}+1-\frac{1}{n-1}.$$ Hence, the claims $1$ and $2$ follow immediately and $3$ follows by Inequality (\ref{ApuEYht}).

Let us then consider a codeword $\bc\in C$ of Type \ref{2}. We may again assume that there is at least one father in $N[\bc]$. Now, the orphan is the only word in $N[\bc]$ which may contribute $1$ to the share, a father contributes at most $1/3$ and all other words contribute at most $1/2$. Therefore, we have $$s(\bc)\leq 1+\frac{1}{3}+\frac{n-1}{2}=\frac{n}{2}+\frac{5}{6}< \frac{n}{2}+1-\frac{1}{n-1}.$$ Hence, the claims$1$ and $2$ follow immediately and $3$ follows with Inequality (\ref{ApuEYht}).

From now on we assume that the codeword $\bc$ has an orphan in its neighbourhood and no codeword neighbours. Consequently, there can be no couples among the words in $N[\bc]$. Let then the codeword $\bc$ be of Type \ref{3}. Since $\bc$ has at most $n-4$ sons, it has at least three fathers in its neighbourhood. Let us first assume that there are $s\geq4$ fathers in $N[\bc]$. Now, we have two words in $N[\bc]$ which give $1$ to the share $s(\bc)$: the orphan and $\bc$ itself. Moreover, since we have at least four fathers, we have at most $n-5$ sons. Finally, each of the fathers contributes at most $1/3$ to the share. Therefore, we have $$s(\bc)\leq 2+\frac{4}{3}+\frac{n-5}{2}=\frac{n}{2}+\frac{5}{6}<\frac{n}{2}+1-\frac{1}{n-1}.$$ Thus, claims $1$, $2$ and $3$ follow in this case and from now on we may assume that we have at most $3$ fathers in $N[\bc]$. Let us then consider the case with exactly three fathers. Recall that we assume $n\ge 10$. If we have exactly three fathers and each of them is a $F_3$-father, then $n=10$ by (\ref{isaMaara}) and we have $s(\bc)\leq 2+(n-4)\frac{1}{2}+3\cdot\frac{1}{3}=\frac{n}{2}+1$ and the claims $1$ and $2$ follow for $n=10$. Moreover, if $n\geq11$, then at least one of the fathers is an $F_i$-father where $i\geq4$ (or we have more than three fathers in $N(\bc)$). Hence, we have \begin{equation}\label{n-4Lasta}
s(\bc)\leq 2+\frac{n-4}{2}+\frac{2}{3}+\frac{1}{4}=\frac{n}{2}+\frac{11}{12}\leq\frac{n}{2}+1+\frac{1}{n^2-5n}-\frac{2}{3n}.
\end{equation}
Therefore, the claims $1$ and $2$ follow immediately and $3$ follows with Inequality (\ref{ApuEYht}).

Let us now assume that the codeword $\bc$ is of Type \ref{4}. Since there is an orphan and no other codewords in its neighbourhood while having $n-3$ sons, it has exactly two fathers $\bu$ and $\bv$ in its neighbourhood. Let us say that they are of types $F_i$ and $F_j$ where $i\leq j$. Now, $s(\bc)\leq 2+(n-3)/2+1/i+1/j$. By Inequality (\ref{isaMaara}), we may assume that $i+j\geq n-1$. Moreover, due to Inequality (\ref{Apuyhtälö}), we obtain $1/i+1/j$ attains its maximum value when $i=3$ and $j\geq n-4$ and, in that case, $1/i+1/j\leq 1/3+1/(n-4)=(n-1)/(3(n-4))$.
If $n\geq16$, then $(n-1)/(3(n-4))\leq\frac{5}{12}$ and $n/2+1/2+(n-1)/(3(n-4))\leq n/2+11/12\leq n/2+(n-2)/(n-1)$ as in (\ref{n-4Lasta}). Moreover, the cases $n\in[10,15]$ are verified by substituting the corresponding value of $n$ to the inequality.
Therefore, we have \begin{equation}\label{n-3Lasta}
s(\bc)\leq2+ \frac{n-3}{2}+\frac{n-1}{3(n-4)}=\frac{n}{2}+\frac{1}{2}+\frac{n-1}{3(n-4)}\leq\begin{cases}
\frac{n}{2}+1, &\text{ if } 10\leq n\leq12\\
\frac{n}{2}+1+\frac{1}{n^2-5n}-\frac{2}{3n}, &\text{ if } 13\leq n.\end{cases}
\end{equation}Therefore, the claims $1$, $2$ and $3$ follow from this result.

Now we may consider the case where the codeword $\bc$ is of Type \ref{5}. Since there is an orphan and $n-2$ sons, there is exactly one $F_{n-1}$-father in $N[\bc]$. Therefore, we have \begin{equation}\label{n-2Lasta}
s(\bc)=2+ \frac{n-2}{2}+\frac{1}{n-1}=\frac{n}{2}+1+\frac{1}{n-1}.
\end{equation}
Therefore, claim $1$, the only claim which concerns this case, follows.

Finally, we have the case where the codeword $\bc$ is of Type \ref{6}. Since there are at least $n-1$ sons and one father, it is impossible to simultaneously have no orphans in $N[\bc]$ and no codewords in $N(\bc)$. Hence, the proof is complete.\end{proof}


We have now shown that if $n \geq10$, then we have $s(\bc)\leq n/2+1+1/(n-1)$ for each $\bc\in C$. However, only a codeword of Type \ref{5} can have share which is greater than $n/2+1$. From now on we will call codewords of Type \ref{5}  \textit{special codewords} and fathers  neighbouring them \textit{special fathers}. Notice that the codeword $\bc$ is special if and only if there are exactly $n-2$ sons, one non-codeword $F_{n-1}$-father and one orphan in $N(\bc)$. A special codeword has share of $n/2+1+1/(n-1)$. Moreover, if a father is special, then it is a non-codeword and an $F_{n-1}$-father. Hence, we inspect special codewords closer and show that their existence means that there exist codewords with smaller share nearby. Thus, we may consider some rules which even out the share among codewords in some selected subset of codewords. Indeed, if we consider Equation (\ref{shareIdea}), then we immediately notice, that we can first even out the share among the selected subset of codewords as long as the total share stays constant (which is $2^n$) and only after that check whether $s(\bc)\leq \alpha$ for each $\bc\in C$. For this we give the following rule.


\begin{enumerate}\label{sääntö1}
\item[\namedlabel{Rule1}{\textbf{Rule} $1$}:] Let $\bc\in C$ be a special codeword and word $\bu$ be the special $F_{n-1}$-father in $N(\bc)$. Moreover, let $C'\subseteq I(\bu)$ be the set of codewords which do not have multiple special fathers in their neighbourhood.
We average out the shares among the codewords in $C'$, that is, the new share for the codewords in $C'$ becomes $$\frac{\sum_{\ba\in C'} s(\ba)}{|C'|}.$$
\end{enumerate}

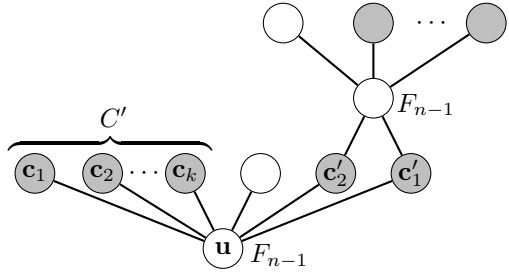
\begin{figure}
\centering
\begin{tikzpicture}


\node[main node](1) at (0,0)    {$ \bu $};
\node[code node](2) at (-2.5,1) {$\bc_1$};
\node[code node](3) at (-1.6,1) {$\bc_2$};
\node[code node](4) at (-0.5,1) {$\bc_k$};
\node[code node](5) at (2.5,1)  {$\bc_1'$};
\node[code node](6) at (1.5,1)  {$\bc_2'$};
\node[main node](7) at (0.5,1)  {};
\node[main node](8) at (2,2)    {};
\node[code node](9) at (3.5,3)  {};
\node[code node](10) at (2.0,3) {};
\node[main node](11) at (0.8,3) {};

\node[] at (-1.02,1) {$\cdots$};
\node[] at (2.8,3) {$\cdots$};
\node[] at (-1.5,1.4) {$\overbrace{\hspace{2.7cm}}$};
\node[] at (-1.45,1.75) {$C'$};
\node[] at (2.7,1.9) {$F_{n-1}$};
\node[] at (0.75,-0.1) {$F_{n-1}$};

\path[draw,thick]
    (1) edge node {} (2)
    (1) edge node {} (3)
    (1) edge node {} (4)
    (1) edge node {} (5)
    (1) edge node {} (6)
    (1) edge node {} (7)
    (8) edge node {} (5)
    (8) edge node {} (6)
    (8) edge node {} (9)
    (8) edge node {} (10)
    (8) edge node {} (11)

    ;

\end{tikzpicture}\centering
\caption{Gray vertices are codewords. Let $\bc=\bc_1$ be special and $\bc_i\in C'$ for $i\in[1,k]$. Now, \ref{Rule1} averages out the share between codewords $\bc_1, \bc_2,\dots,\bc_k$.}\label{Rule1Esim.}
\end{figure}

We have illustrated \ref{Rule1} in Figure \ref{Rule1Esim.}. Notice that if $\bc$ is a special codeword, then it has exactly one $F_{n-1}$-father in $N[\bc]$ and that father is special. Moreover, the set $C'$ consists of those codewords which have exactly one special father in their neighbourhood and hence, each codeword takes part in at most one averaging process when we apply \ref{Rule1}. Next, we present the lower bound which we can obtain with \ref{Rule1}. This bound is already an improvement to the previous lower bound presented in Theorem \ref{VanhaRaja}.


\begin{theorem}\label{perusraja}
Let $n\geq10$ and $C$ be a locating-dominating code. After applying \ref{Rule1}, each codeword $\bc\in C$ has $$s(\bc)\leq \frac{n}{2}+1.$$ 
\end{theorem}
\begin{proof}

Let $C$ be a locating-dominating code in $\F^n$ and $n\geq10$. If there are no special codewords, then we are ready since $s(\bc)\leq n/2+1$ by Lemma \ref{TyyppiLemma}. Assume then that there exists at least one special codeword.
Let us assume that the all-zero word $\nolla$ is a special $F_{n-1}$-father, denote $\bu=\nolla$, and $\e_1\in C$ is a special codeword (in Figure \ref{Rule1Esim.} codeword $\bc_1$ corresponds to $\e_1$). Hence, $\bu\not\in C$ and exactly one of words $\e_i$, $i\geq2$, is a non-codeword. Let us say $\e_n\not\in C$. In addition, since $\bu\not\in C$ and $I(\e_n)\neq \emptyset$, at least one of the words $\e_n+\e_i$, $i\in [2,n-1]$, is a codeword, say $\e_2+\e_n\in C$ and denote $\bv=\e_2+\e_n$. Notice that $N[\e_2]=\{\bu, \e_2,\e_2+\e_1,\e_2+\e_3,\dots, \e_2+\e_n\}$ and thus, there are no orphans in $N[\e_2]$ and there is a codeword $\bv\in C$ next to $\e_2$. Now, we have \begin{equation}\label{e_2share}
s(\e_2)\leq \frac{1}{n-1}+\frac{2}{2}+\frac{n-2}{2}=\frac{n}{2}+\frac{1}{n-1}.
\end{equation} Moreover, let us denote by $C'\subseteq I(\bu)$ the set of codewords without multiple neighbouring special fathers. Notably $\e_1\in C'$ and the same is true for all other special codewords in $I(\bu)$ if there are any. Let us first assume that $\e_2\in C'$. Thus, Lemma \ref{TyyppiLemma} gives $$\sum_{\e_i\in C'}s(\e_i)\leq  (|C'|-1)\left(\frac{n}{2}+1+\frac{1}{n-1}\right)+\frac{n}{2}+\frac{1}{n-1}\leq|C'|\left(\frac{n}{2}+1\right).$$ Hence, if we apply \ref{Rule1} where $\bc$ corresponds to $\e_1$ and $\bu$ to $\nolla$, then averaging out the shares gives share of at most $\frac{n}{2}+1$ for each codeword in $C'$. 


Next we finalize the proof by showing that $\e_2$ always belongs to $C'$. Assume to the contrary that $\e_2\notin C'$.
Let $\e_2+\e_i$ be a special $F_{n-1}$-father other than $\nolla$. Note that $i\neq 1$ since $\e_1$ is special and $i\neq2$ since $\e_2+\e_i\neq\bu$. Moreover, if $i=n$, then $\e_2+\e_i=\bv$ is a codeword and thus, it cannot be a special father, a contradiction. Hence, we may assume that $3\leq i\leq n-1$. Without loss of generality, let $i=3$. We also notice that $\e_1+\e_2+\e_3\not\in C$ since $\e_1$ is special and sons in $N(\e_1)$ have their two codeword neighbours in $N(\bu)$. Now, we have $I(\e_2+\e_3)=\{\e_2,\e_3, \e_2+\e_3+\e_4,\dots,\e_2+\e_3+\e_{n}\}$. In the following, we show that there are no special codewords in $I(\e_2+\e_3)$ and hence, $\e_2+\e_3$ is not a special father. First of all, $\e_2$ and $\e_3$ are not special due to the fact that they have two $F_{n-1}$-fathers in their neighbourhood. Moreover, we have $\{\e_2,\e_j,\e_2+\e_3+\e_j\}\subseteq I(\e_2+\e_j)$ for each $j\in [4,n-1]$. Hence, none of codewords $\e_2+\e_3+\e_j$ is special when $j\in[4,n-1]$ since they have multiple adjacent fathers. Finally, we have $\e_2+\e_n\in I(\e_2+\e_3+\e_n)$ and hence, $\e_2+\e_3+\e_n$ is not a special codeword. Therefore, we do not have any special codewords in $N[\e_2+\e_3]$ and $\e_2+\e_3$ is not a special father. Hence, we have a contradiction and, thus, $\e_2\in C'$. This completes the proof.\end{proof}

The next result follows now immediately from Equation (\ref{shareIdea2}).

\begin{corollary}\label{perusCorollary}
Let $n\geq10$. We have $$\gamma^{LD}(\F^n)\geq \frac{2^{n+1}}{n+2}.$$
\end{corollary}

\section{The second bound}

Observe that there are still some codewords which have share less than $n/2+1$, namely the codeword $\bv=\e_2+\e_n$, which is of Type \ref{2}, and all other codewords of Types \ref{1}, \ref{2}, \ref{3}, \ref{4} and \ref{6} also have smaller shares when $n\geq11$. Now our goal is to show the following theorem.
\begin{theorem}\label{parasraja}
We have $$\gamma^{LD}(\F^n)\geq \begin{cases}\frac{2^{n+1}}{n+1+2(n-1)/(3(n-4))}, &\text{if } 11\leq n\leq12\\
\frac{2^{n+1}}{n+2+2/(n^2-5n)-4/(3n)},&\text{if }13\leq n.\end{cases}$$
\end{theorem}

To prove this theorem, we have to consider the locating-dominating code and the Hamming space in more detail. Again $\nolla$ is a special $F_{n-1}$-father, $\e_1\in C$ is a special codeword, $\e_n\not\in C$ and $\bv=\e_2+\e_n\in C$. We will show that after some share shifting, we have $s(\bc)\leq n/2+1+1/(n^2-5n)-2/(3n)$ which gives Theorem \ref{parasraja} when $n\geq13$. Lemma \ref{TyyppiLemma} has already shown that when $n\geq13$, only the special codewords have share greater than $n/2+1+1/(n^2-5n)-2/(3n)$. We will proceed by first showing that if we have multiple $F_{n-1}$-fathers 
in $B_2(\nolla)$, then, after applying \ref{Rule1}, we have $s(\bc)<n/2+1-1/(n-1)$, which is less than $n/2+1+1/(n^2-5n)-2/(3n)$ for each $\bc\in I(\nolla)$ by (\ref{ApuEYht}). After that we use similar deduction to show that $|I_3(\nolla)|=n$. Finally, we implement some rules to shift share into $\bv$ and out of $\bv$.

\begin{lemma}\label{Ein-1isiäB2ssa}
Let $C\subseteq \F^n$ and $n\geq11$. Let $\bu$ and $\bx$ special fathers, $d(\bu,\bx)\leq 2$ and $\bc$ be a special codeword in $I(\bu)$. We have $s(\bc)\leq n/2+1-1/(n-1)$ after applying \ref{Rule1}.
\begin{proof}
Let $\bu=\nolla$ and $\bx\neq \nolla$ be special $F_{n-1}$-fathers in $B_2(\nolla)$. Hence, $\bu,\bx\not\in C$. Moreover, let us assume, without loss of generality, that $\e_1$ is a special codeword, $\e_n$ is the non-codeword in $N(\bu)$ and $\e_2+\e_n\in C$. As we have seen in the proof of Theorem \ref{perusraja}, $\bx$ cannot be in $N[\e_1]$ or $N[\e_2]$ and hence, $\e_1,\e_2\in C'$. Moreover, if $\bx=\e_n$, then $\e_1+\e_n\in C$ since $\nolla\not\in C$. However, this is not possible, since $\e_1$ is special. Hence, we have $w(\bx)=2$. If we have $\bx=\e_i+\e_n$, where $3\leq i\leq n-1$, then $\e_1+\e_i+\e_n\in C$ since $\e_n\not\in C$. Again, this is not possible since $\e_1$ is a special codeword and $\e_1+\e_i$ is a son. Therefore, we have $$\bx=\e_i+\e_j$$ where $3\leq i,j\leq n-1$ and $i\neq j$. Without loss of generality, let us say $i=3$ and $j=4$. Moreover, $I(\e_3+\e_4)=\{\e_3+\e_4+\e_h\mid 2\leq h\leq n\}$ since $\e_1$ is special and $\e_1+\e_3$ is a son. Furthermore, we have a special codeword in $I(\bx)$. However, codewords $\e_3$ or $\e_4$ cannot be special since they have two $F_{n-1}$-fathers in their neighbourhoods. Moreover, codewords $\e_3+\e_4+\e_h$, $h\in [2,n-1]$ and $h\neq 3,4$, cannot be special since $\e_3+\e_h\in N[\e_3+\e_4+\e_h]$ and $\{\e_3,\e_h,\e_3+\e_4+\e_h\}\subseteq I(\e_3+\e_h)$. Therefore, $\e_3+\e_4+\e_n$ is special.

Let $C'\subseteq I(\nolla)$ be the set of codewords which do not have multiple neighbouring special fathers. Notice that $\e_3,\e_4\not\in C'$. We have $\e_1,\e_2\in C'$ as we have seen above. If $\e_t\in C'$ for some $t\in[5,n-1]$, then we have $\{\e_t,\e_3,\e_3+\e_4+\e_t\}\subseteq I(\e_t+\e_3)$. Similarly, we get $|I(\e_t+\e_4)|\geq3$. Thus, $$s(\e_t)\leq \frac{1}{n-1}+2\cdot\frac{1}{3}+2+\frac{n-4}{2}=\frac{n}{2}+\frac{2}{3}+\frac{1}{n-1}< \frac{n}{2}+1-\frac{1}{n-1}.$$ Moreover, since $\e_2+\e_3$ and $\e_2+\e_4$ are fathers and $\e_2+\e_n\in C$, we have $$s(\e_2)\leq \frac{1}{n-1}+2\cdot\frac{1}{2}+2\frac{1}{3}+\frac{n-4}{2}=\frac{n}{2}-\frac{1}{3}+\frac{1}{n-1}.$$ Therefore, we have \begin{align*}\sum_{\bc'\in C'}s(\bc')\leq &\left(\frac{n}{2}+1+\frac{1}{n-1}\right)+\left(\frac{n}{2}-\frac{1}{3}+\frac{1}{n-1}\right)+(|C'|-2)\left(\frac{n}{2}+1-\frac{1}{n-1}\right)\\<&|C'|\left(\frac{n}{2}+1-\frac{1}{n-1}\right).
\end{align*}\end{proof}
\end{lemma}

From now on, we may assume that no codeword in $I(\nolla)$ has multiple special fathers in its neighbourhood. In particular, this makes applying \ref{Rule1} easier. Moreover, in the following we show, that we may assume that there are no codewords in $B_3(\nolla)$ except for $\bv=\e_2+\e_n$ and those in $I(\nolla)$.

\begin{lemma}\label{nKoodisanaaB3:ssa}
Let $n\geq11$, $\bu$ be a special father and $\bc\in I(\bu)$ be a special codeword. If $|I_3(\bu)|\geq n+1$, then we have $s(\bc)\leq n/2+1-1/(n-1)$ after applying \ref{Rule1}.
\begin{proof}
Let $\bu=\nolla$ and let $\bc'\in B_3(\nolla)$ be a codeword other than $\nolla$, any codeword in $I(\nolla)$ or $\bv$. Since $\e_1$ is a special codeword and $\nolla$ is a special father, we have seven possible cases for $\bc'$ as $w(\bc')\geq2$ and $d(\bc',\e_1)\geq3$. In particular, there are three different possibilities for $\bc'$ with $w(\bc')=2$ (see Cases $I-III$ below) and four different possibilities for $\bc'$ with $w(\bc')=3$ (Cases $IV-VII$). Moreover, we may assume by Lemma \ref{Ein-1isiäB2ssa} that no codeword in $I(\nolla)$ has multiple special fathers in its neighbourhood, $s(\e_q)\leq n/2+1+1/(n-1)$, for $q\neq 2,n$, and $$s(\e_2)\leq n/2+1/(n-1)$$ as seen in Inequality (\ref{e_2share}). Since additional codewords in $I_3(\nolla)$ cannot increase the share of codewords in $I(\nolla)$, we assume that there are exactly $n+1$ codewords in $I_3(\nolla)$. In the following, distinct indices $i,j,h$ belong to the set $[3,n-1]$: \begin{enumerate}[label=\Roman*.]
\item $\e_2+\e_i=\bc'$
\item $\e_n+\e_i=\bc'$
\item $\e_i+\e_j=\bc'$
\item $\e_2+\e_n+\e_i=\bc'$
\item $\e_2+\e_i+\e_j=\bc'$
\item $\e_n+\e_i+\e_j=\bc'$
\item $\e_i+\e_j+\e_h=\bc'$
\end{enumerate}

In  Case I, $\e_2$ and $\e_2+\e_i$ are $F_3$-fathers.  Thus, $$s(\e_2)\leq\frac{1}{n-1}+\frac{2}{3}+\frac{n-2}{2}=\frac{n}{2}-\frac{1}{3}+\frac{1}{n-1}$$ and since $\e_i$ is a son, we have $$s(\e_i)\leq \frac{1}{n-1}+\frac{1}{3}+1+\frac{n-2}{2}=\frac{n}{2}+\frac{1}{3}+\frac{1}{n-1}.$$ Recall, that by our assumptions $C'=\{\e_1,\dots,\e_{n-1}\}$ in \ref{Rule1}. Therefore, we have \begin{align*}
\sum_{k=1}^{n-1}s(\e_k)\leq& (n-3)\left(\frac{n}{2}+1+\frac{1}{n-1}\right)+\left(\frac{n}{2}+\frac{1}{3}+\frac{1}{n-1}\right)+\left(\frac{n}{2}-\frac{1}{3}+\frac{1}{n-1}\right)\\
=& (n-1)\left(\frac{n}{2}+1-\frac{1}{n-1}\right)
\end{align*}
and we are ready.

In Case II, $\e_i$ and $\e_n+\e_i$ form a couple. Hence,
$$s(\e_i)\leq \frac{1}{n-1}+\frac{2}{2}+\frac{n-2}{2}=\frac{n}{2}+\frac{1}{n-1}$$ and \begin{align*}
\sum_{k=1}^{n-1}s(\e_k)\leq& (n-3)\left(\frac{n}{2}+1+\frac{1}{n-1}\right)+2\left(\frac{n}{2}+\frac{1}{n-1}\right)\\
=& (n-1)\left(\frac{n}{2}+1-\frac{1}{n-1}\right).
\end{align*}

In Case III, $\e_i$ and $\e_j$ are sons and $\e_i+\e_j$ is an $F_3$-father. We have $$s(\e_i)\leq \frac{1}{n-1}+\frac{1}{2}+\frac{1}{3}+1+\frac{n-3}{2}=\frac{n}{2}+\frac{1}{3}+\frac{1}{n-1}$$  \text{ and same is true for } $\e_j$. Hence, \begin{align*}
\sum_{k=1}^{n-1}s(\e_k)\leq& (n-4)\left(\frac{n}{2}+1+\frac{1}{n-1}\right)+2\left(\frac{n}{2}+\frac{1}{3}+\frac{1}{n-1}\right)+\left(\frac{n}{2}+\frac{1}{n-1}\right)\\
=& (n-1)\left(\frac{n}{2}+1-\frac{1}{n-1}-\frac{1}{3(n-1)}\right).
\end{align*}

In Case IV, $\e_2+\e_n$ and $\e_2+\e_i$ are $F_3$-fathers. Thus, we have
$$s(\e_2)\leq \frac{1}{n-1}+\frac{2}{3}+\frac{1}{2}+\frac{n-3}{2}=\frac{n}{2}-\frac{1}{3}+\frac{1}{n-1}$$ and since $\e_2+\e_i$ is an $F_3$-father and $\e_n+\e_i$ is a son, we have
$$s(\e_i)\leq \frac{1}{n-1}+\frac{1}{3}+\frac{1}{2}+1+\frac{n-3}{2}=\frac{n}{2}+\frac{1}{3}+\frac{1}{n-1}.$$
Hence,
\begin{align*}
\sum_{k=1}^{n-1}s(\e_k)\leq& (n-3)\left(\frac{n}{2}+1+\frac{1}{n-1}\right)+\left(\frac{n}{2}+\frac{1}{3}+\frac{1}{n-1}\right)+\left(\frac{n}{2}-\frac{1}{3}+\frac{1}{n-1}\right)\\
=& (n-1)\left(\frac{n}{2}+1-\frac{1}{n-1}\right).
\end{align*}

In Case V, $\e_2+\e_i$, $\e_2+\e_j$ and $\e_i+\e_j$ are $F_3$-fathers. Thus, we have
$$s(\e_2)\leq \frac{1}{n-1}+\frac{2}{3}+\frac{2}{2}+\frac{n-4}{2}=\frac{n}{2}-\frac{1}{3}+\frac{1}{n-1},$$
$$s(\e_i)\leq \frac{1}{n-1}+\frac{2}{3}+2+\frac{n-4}{2}=\frac{n}{2}+\frac{2}{3}+\frac{1}{n-1}$$ and the same is true for $s(\e_j)$.
Hence,
\begin{align*}
\sum_{k=1}^{n-1}s(\e_k)\leq& (n-4)\left(\frac{n}{2}+1+\frac{1}{n-1}\right)+2\left(\frac{n}{2}+\frac{2}{3}+\frac{1}{n-1}\right)+\left(\frac{n}{2}-\frac{1}{3}+\frac{1}{n-1}\right)\\
=& (n-1)\left(\frac{n}{2}+1-\frac{1}{n-1}\right).
\end{align*}

In Case VI, $\e_i+\e_j$ is an $F_3$-father and $\e_i+\e_n$ and $\e_j+\e_n$ are sons. Thus, we have
$$s(\e_i)\leq \frac{1}{n-1}+\frac{1}{3}+\frac{1}{2}+1+\frac{n-3}{2}=\frac{n}{2}+\frac{1}{3}+\frac{1}{n-1}$$
and the same is true for $s(\e_j)$.
Hence,
\begin{align*}
\sum_{k=1}^{n-1}s(\e_k)\leq& (n-4)\left(\frac{n}{2}+1+\frac{1}{n-1}\right)+2\left(\frac{n}{2}+\frac{1}{3}+\frac{1}{n-1}\right)+\left(\frac{n}{2}+\frac{1}{n-1}\right)\\
=& (n-1)\left(\frac{n}{2}+\frac{n-2-\frac{1}{3}}{n-1}\right).
\end{align*}

In Case VII, $\e_i+\e_j$, $\e_i+\e_h$ and $\e_j+\e_h$ are $F_3$-fathers. Thus, we have
$$s(\e_i)\leq \frac{1}{n-1}+\frac{2}{3}+2+\frac{n-4}{2}=\frac{n}{2}+\frac{2}{3}+\frac{1}{n-1}$$
and the same is true for $s(\e_j)$ and $s(\e_h)$.
Hence,
\begin{align*}
\sum_{k=1}^{n-1}s(\e_k)\leq& (n-5)\left(\frac{n}{2}+1+\frac{1}{n-1}\right)+3\left(\frac{n}{2}+\frac{2}{3}+\frac{1}{n-1}\right)+\left(\frac{n}{2}+\frac{1}{n-1}\right)\\
=& (n-1)\left(\frac{n}{2}+1-\frac{1}{n-1}\right).
\end{align*}

\end{proof}
\end{lemma}

From now on, we call a special father $\bu\in\F^n$ with $|I_3(\bu)|=n$ a \textit{sparse father}. Observe that if $\bu$ is a sparse father, then $\bu\not\in C$, there are exactly $n-1$ codewords in $I(\bu)$ and $n-2$ of these are special. Indeed, all the $n-2$ codewords in $I_3(u)$ other than the two codewords forming a couple are special codewords. Moreover, there is exactly one codeword at distance two from $\bu$ and no codewords at distance three from $\bu$. As can be seen in Figure \ref{Rule3Esim.}.

Now that we have restricted the structure of $C$ in $B_3(\nolla)$, we are ready to describe new rules for shifting share away from $\bv$ and, later, into $\bv$. \ref{Rule22} is applied after we have applied \ref{Rule1} to each suitable codeword in $C$.

\begin{enumerate}\label{sääntö22}
\item[\namedlabel{Rule22}{\textbf{Rule} $2$}:] Let $\bu$ be a sparse father and $\bv\in C$ be the codeword at distance two from $\bu$, that is, $d(\bu,\bv)=2$. Moreover, let there be exactly one $F_{n-2}$-father, say $\ba$, in $N(\bv)$. 
If, after applying \ref{Rule1}, we have  $s(\bv)> n/2+1/3+1/(n-5)$, then we shift $$\frac{1/6+1/(n-2)-1/(n-5)}{n-3}$$ share from $\bv$ to each other codeword in $I(\ba)$.\end{enumerate}

\begin{figure}
\centering
\begin{tikzpicture}


\node[main node](1) at (0,0)    {$\bu$};
\node[code node](2) at (-2.5,1) {};
\node[code node](3) at (-1.6,1) {};
\node[code node](4) at (-0.5,1) {};
\node[main node](5) at (2.5,1)  {};
\node[code node](6) at (1.5,1)  {};
\node[code node](7) at (0.5,1)  {};
\node[main node](18) at (0.8,2)  {};
\node[main node](28) at (1.5,2)  {};
\node[code node](8) at (2.2,2)    {$\bv$};
\node[main node](11) at (1.5,3) {$\ba$};

\node[code node](12) at (0,4) {};
\node[code node](13) at (1,4) {};
\node[code node](14) at (2,4) {};
\node[code node](15) at (3,4) {};

\node[] at (0.03,1) {$\cdots$};
\node[] at (2.53,4) {$\cdots$};
\node[] at (1.5,4.35) {$\overbrace{\hspace{3.5cm}}$};
\node[] at (1.5,4.6) {$n-3$ codewords};
\node[] at (2.2,2.95) {$F_{n-2}$};
\node[] at (0.75,-0.1) {$F_{n-1}$};

\path[draw,thick]
    (1) edge node {} (2)
    (1) edge node {} (3)
    (1) edge node {} (4)
    (1) edge node {} (5)
    (1) edge node {} (6)
    (1) edge node {} (7)
    (8) edge node {} (5)
    (28) edge node {} (5)
    (18) edge node {} (6)
    (28) edge node {} (7)
    (18) edge node {} (7)
    (8) edge node {} (6)
    (11) edge node {} (18)
    (11) edge node {} (28)
    (8) edge node {} (11)
    (11) edge node {} (12)
    (11) edge node {} (13)
    (11) edge node {} (14)
    (11) edge node {} (15)

    ;

\end{tikzpicture}\centering
\caption{Gray vertices are codewords. Word $\bu$ is a sparse father and word $\ba$ is an $F_{n-2}$-father. Now, \ref{Rule22} shifts $(1/6+1/(n-2)-1/(n-5))/(n-3)$ share from codeword $\bv$ to each other codeword in $I(\ba)$.}\label{Rule2Esim.}
\end{figure}
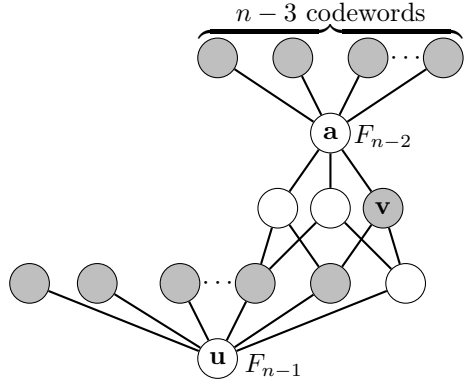

We have illustrated \ref{Rule22} in Figure \ref{Rule2Esim.}. Observe that if the codeword $\bv$ has another father besides the $F_{n-2}$-father in \ref{Rule22}, then the codeword has share of $s(\bv)\leq \frac{1}{n-2}+\frac{1}{3}+\frac{n}{2}$ and \ref{Rule22} does not shift any share. Before approximating how \ref{Rule22} changes the shares, we show that neither \ref{Rule22} nor \ref{Rule1} shift share to word $\bv$.

\begin{lemma}\label{EiShareaVlle}
Let $n\geq 11$. Then \ref{Rule1} or \ref{Rule22} do not shift share to the codeword $\bv$ at distance two from a sparse father $\bu$.
\begin{proof}
Let us again assume, without loss of generality, that $\bu=\nolla\not\in C$ is a sparse father, $\e_n\not\in C$ and $\bv=\e_2+\e_n\in C$.

Let us first assume that \ref{Rule1} shifts share to $\bv$. Hence, we have another $F_{n-1}$-father in $B_3(\nolla)$ and there are at least $n+1$ codewords in $B_3(\nolla)$, a contradiction. Let us assume then that \ref{Rule22} shifts share to $\bv$. Since  \ref{Rule22} shifts share to $\bv$, there is an $F_{n-2}$-father in $N(\bv)$ of weight three. Let us assume without loss of generality that $\ba=\e_2+\e_3+\e_n\not\in C$ is the $F_{n-2}$-father. Observe that the only non-codewords in $N[\ba]$ are $\e_2+\e_3$, $\e_3+\e_n$ and $\ba$ itself. We may assume that we shift the share from word $\bw=\ba+\e_i\in I(\ba)$ where $i\in\{1\}\cup[4,n-1]$ to $\bv$. Hence, there has to be a special father $\bu'$ in $B_2(\bw)$ and another codeword in $I(\bw)$. Observe that $w(\bu')=6$. Indeed, since $w(\bw)=4$ and $d(\bw,\bu')=2$, we have $w(\bu')\in\{2,4,6\}$ and since $|I_3(\nolla)|=n$, we have $w(\bu')=6$. Thus, there exists a codeword $\bw+\e_j\in I(\bw)\cap I(\bu')$ of weight five. Since $d(\bw+\e_j,\ba)=2$, they have two common neighbours and both of them belong to the code $C$ since the non-codewords in $N(\ba)$ have weight of two. However, both of these words locate in $B_3(\bu')\setminus N[\bu']$ and thus, $|I_3(\bu')|\geq n+1$. Hence, \ref{Rule22} does not shift share from $\bw$ to $\bv$ (a contradiction) and thus, the claim follows.
\end{proof}
\end{lemma}

Now, in Lemma \ref{22 siirto} we approximate how \ref{Rule22} changes the amounts of share.

\begin{lemma}\label{22 siirto}
Let $n\geq 11$ and let $\ba$, $\bu$ and $\bv$ be as in \ref{Rule22}. Then 
\begin{enumerate}
\item[1.] after applying \ref{Rule1} and \ref{Rule22} we have $$s(\bc)\leq \begin{cases}\frac{n}{2}+\frac{1}{2}+\frac{n-1}{3(n-4)},&\text{if } n\in[11,12]\\
\frac{n}{2}+1+\frac{1}{n^2-5n}-\frac{2}{3n},&\text{if } 13\leq n\end{cases}$$ for each $\bc\in I(\ba)\setminus\{\bv\}$ and
\item[2.] after applying \ref{Rule1} and \ref{Rule22}, $s(\bv)\leq n/2+1/3+1/(n-5).$
\end{enumerate}
\begin{proof}
Let us again assume, without loss of generality, that $\bu=\nolla\not\in C$ is a sparse father, $\e_n\not\in C$ and $\bv=\e_2+\e_n\in C$. Moreover, we may assume that $s(\bv)> n/2+1/3+1/(n-5)$ as otherwise the second claim is immediately clear and the first one follows by Lemma~\ref{TyyppiLemma}.

Let us again assume, without loss of generality, that $\ba=\e_2+\e_3+\e_n\not\in C$ is the only $F_{n-2}$ father in $N(\bv)$.  Let us then consider Case $1$. Let $\bc\in I(\ba)\setminus\{\bv\}=\{\ba+\e_i\mid i=1 \text{ or } i\in[4,n-1]\}$. Indeed, notice that the non-codewords in $N(\ba)$ have weight two since $\nolla$ is a sparse father and thus, each weight four neighbour of $\ba$ is a codeword. We may assume, without loss of generality, that $\bc=\ba+\e_1=\e_1+\e_2+\e_3+\e_n$. (Observe that unlike, for example, in the proof of Theorem \ref{perusraja}, we have not assigned any role for word $\e_1$ and we may now make the assumption on $\bc$.) Now we further split into the following three subcases:
\begin{enumerate}
\item[A.] First we assume that $\bc$ does not have exactly one adjacent special father or other adjacent $F_{n-2}$-fathers except $\ba$,
\item[B.] then we assume that there are multiple $F_{n-2}$-fathers in $N(\bc)$ and $\bc$ does not have  exactly one adjacent special father and
\item[C.] finally, we assume that there is exactly one special father in $N(\bc)$ and possibly multiple $F_{n-2}$-fathers.
\end{enumerate}
Together, these three subcases go through all the possibilities. Observe, that by the definition \ref{Rule1} is applied to the codeword $\bc$ only if there is exactly one adjacent special father.

\medskip


\emph{Let us now consider Subcase $A$.} We assume that the only $F_{n-2}$-father in $N(\bc)$ is $\ba$ and there are either no special fathers in $N(\bc)$ or there are multiple special fathers. Hence, \ref{Rule1} does not shift share to or from $\bc$. Observe that $I(\e_1+\e_2+\e_3)\cap N(\ba)=\{\bc\}$ and $I(\e_1+\e_3+\e_n)\cap N(\ba)=\{\bc\}$. Since $\e_1+\e_2+\e_3$ and $\e_1+\e_3+\e_n$ are non-codewords as $\bu$ is a sparse father and $C$ is a locating-dominating code, there has to be a codeword $\bc'$ ($\neq\bc$) such that $|N(\bc')\cap\{\e_1+\e_3+\e_n,\e_1+\e_2+\e_3\}|=1$ and $w(\bc')=4$. Hence, we have $d(\bc,\bc')=2$  and $\bc'\not\in I(\ba)$, that is, $\bc'=\e_1+\e_3+\e_x+\e_i$ where $x\in\{2,n\}$ and $i\in [4,n-1]$. Moreover, by Lemma \ref{palloleikkaus}, there has to be another common neighbour $\bb$ with $\bc$ and $\bc'$ besides $\bc'+\e_i$. Furthermore, we have $w(\bb)=5$ and $|I(\bb)|\geq3$ since $d(\ba,\bb)=2$ and $N(\ba)\cap N(\bb)\subseteq C$ as each neighbour of $\ba$ of weight four is a codeword. Thus $\{\bc'\}\cup(N(\ba)\cap N(\bb))\subseteq I(\bb)$. Therefore, we have $$s(\bc)\leq \frac{1}{n-2}+\frac{1}{3}+2+\frac{n-3}{2}=\frac{n}{2}+\frac{5}{6}+\frac{1}{n-2}$$ before applying \ref{Rule22}. Moreover, after applying \ref{Rule22} we have $$s(\bc)\leq \frac{n}{2}+\frac{5}{6}+\frac{1}{n-2}+\frac{1/6+1/(n-2)-1/(n-5)}{n-3}\leq\begin{cases}\frac{n}{2}+\frac{1}{2}+\frac{n-1}{3(n-4)}, &\text{if } 11\leq n\leq 12\\
\frac{n}{2}+1+\frac{1}{n^2-5n}-\frac{2}{3n}, &\text{if } 13\leq n. \end{cases}$$Indeed, the inequality can be shown as follows. The cases $n\in[11,13]$ are verified by substituting the corresponding value of $n$ to the inequalities. When $n\geq14$, we observe that $1/(n-2)+(1/6+1/(n-2)-1/(n-5))/(n-3)<1/(n-2)+1/(6(n-3))<1/10$ and $1/6+1/(n^2-5n)-2/(3n)>1/6-2/(3n)>1/10$.

\medskip

\emph{Let us then consider Subcase $B$} where we have $2\leq h$ copies of $F_{n-2}$-fathers in $N(\bc)$ and we do not have exactly one special father in $N(\bc)$. Hence, we shift share to codeword $\bc$ with \ref{Rule22} at most $h$ times and \ref{Rule1} does not shift share to or from the codeword $\bc$. Moreover, we have $h\leq n-3$ since 
$\bc$ has four neighbours of weight $3$ and only one of these can be a non-codeword $F_{n-2}$-father, that is, the word $\ba$. Indeed, this is true since we have $|I_3(\nolla)|=n$, each non-codeword $F_{n-2}$-father of weight three has a codeword neighbour of weight two and the only codeword of weight two, $\bv$, has exactly one $F_{n-2}$-father in its neighbourhood.
Hence, we may give a (rough) approximate $$s(\bc)\leq \frac{h}{n-2}+2+\frac{n-1-h}{2}=\frac{n}{2}+\frac{3}{2}+\frac{h}{n-2}-\frac{h}{2}\leq \frac{n}{2}+\frac{1}{2}+\frac{2}{n-2}$$ before applying \ref{Rule22}. Observe that since $h\leq n-3$, we add at most $1/6$ share to $\bc$ and after applying \ref{Rule22} we have $$s(\bc)\leq \frac{n}{2}+\frac{2}{3}+\frac{2}{n-2}\leq \begin{cases}\frac{n}{2}+\frac{1}{2}+\frac{n-1}{3(n-4)}, &\text{if } 11\leq n\leq 12\\
\frac{n}{2}+1+\frac{1}{n^2-5n}-\frac{2}{3n}, &\text{if } 13\leq n. \end{cases}$$
Indeed, we immediately see that this is true by comparing this to the inequality in Subcase $A$  since $n/2+2/3+2/(n-2)<n/2+5/6+1/(n-2)$. 

\medskip

\emph{Then there is Subcase $C.$} In this case there is exactly one special father in $N(\bc)$ and  $h$ copies of $F_{n-2}$-fathers where $h\in[1,n-3]$ as above. Now, we first shift share using \ref{Rule1} to $\bc$ and after that at most $h$ times (but at least once) using \ref{Rule22}.  Moreover, let $\bw$ be the special father in $N(\bc)$. 
Let us again denote by $C'$ the set of codewords in $I(\bw)$ with exactly one special father in their neighbourhoods.  Now, as we will see next, there can be at most two special codewords in $I(\bw)$. Furthermore, we have $F_{n-2}$-father $\ba$ at distance two from the special father $\bw$ and hence, there are multiple codewords at distance three from father $\bw$. Observe that $w(\bw)=5$ since $\bu$ is a sparse father and $d(\bw,\ba)=2$. Recall that all words of weight four in $N(\ba)$ are codewords since $\ba$ is an $F_{n-2}$-father. Hence, $|I(\ba)\cap I(\bw)|=2$. Recall also that we have assumed  $s(\bv)>n/2+1/3+1/(n-5)$.

\smallskip

We may assume without loss of generality that $\bw=\ba+\e_1+\e_4=\bc+\e_4=\e_1+\e_2+\e_3+\e_4+\e_n$. Let us first show that $\bw+\e_3\not\in C$. If $\bw+\e_3\in C$, then $\{\bc,\bw+\e_3,\bv\}\subseteq I(\bc+\e_3)$. Hence, we have $s(\bv)\leq1/3+1/(n-2)+1+2/2+(n-4)/2=n/2+1/3+1/(n-2)<n/2+1/3+1/(n-5)$ a contradiction with one of the assumptions. Thus, $I(\bw)=N(\bw)\setminus\{\bw+\e_3\}$.

 Next, we show that there are at most two special codewords in $I(\bw)$. First of all, $\bw+\e_1$ and $\bw+\e_4$ are adjacent to $\ba$ and therefore are not special. Moreover, there are at least two fathers in $N(\bw+\e_i)$, for $i\in\cup[5,n-1]$. Indeed, we have \begin{equation}\label{w naapurit}
\{\bw+\e_i,\bw+\e_4,\ba+\e_i\}\subseteq I(\bw+\e_i+\e_4).
\end{equation} Thus, the codeword $\bw+\e_i$ is not special and the only possible special codewords are $\bw+\e_2$ and $\bw+\e_n$.

In the following, we give a rough estimate for the share of codewords of type $\bw+\e_i$, $i\in[5,n-1]$. Notice that in (\ref{w naapurit}) we have shown that $\bw+\e_i+\e_4$ is a father. Similarly, we can see that $\bw+\e_i+\e_1$ is a father. Hence, we get approximation $$s(\bw+\e_i)\leq 2+\frac{2}{3}+\frac{1}{n-1}+\frac{n-4}{2}=\frac{n}{2}+\frac{2}{3}+\frac{1}{n-1}.$$ Moreover, we may use this upper bound also for $\bw+\e_1$ since $\ba\in N(\bw+\e_1)$ and $1/(n-2)+1/2<2/3$ as $n-2>6$.

Let us then consider the neighbourhood of $\bc$. As we have seen in (\ref{w naapurit}), $\bc=\bw+\e_4$ has $n-5$ adjacent fathers of form $\bc+\e_i=\bw+\e_i+\e_1$, $i\in [5,n-1]$. Moreover, $\bc+\e_1=\ba$ is an $F_{n-2}$-father and $\bc+\e_4=\bw$  is an $F_{n-1}$-father. Furthermore, $\bc+\e_2$, $\bc+\e_n$ and $\bc+\e_3$ are adjacent to the codewords $\bw+\e_2$, $\bw+\e_n$ and $\bv$, respectively. Thus, we may have at most three sons, no orphans and at least $n-3$ fathers in $N(\bc)$. Moreover, at least $h$ of these $n-3$ fathers are $F_{n-2}$-fathers. Therefore, we get $$s(\bc)\leq \frac{1}{n-1}+\frac{h}{n-2}+\frac{3}{2}+1+\frac{n-4-h}{3}=\frac{n-h}{3}+\frac{7}{6}+\frac{h}{n-2}+\frac{1}{n-1}.$$

 By these considerations we get the following (rough) upper bound for the share $s(\bc)$ after applying \ref{Rule1}. The first term of the sum is for the special codewords, second for $\bc$ and the third one for other codewords in $C'$.
\begin{align*}
|C'|s(\bc)&\leq 2\left(\frac{n}{2}+1+\frac{1}{n-1}\right)+\left(\frac{n-h}{3}+\frac{7}{6}+\frac{h}{n-2}+\frac{1}{n-1}\right)+\left(|C'|-3\right)\left(\frac{n}{2}+\frac{2}{3}+\frac{1}{n-1}\right)\\
&=|C'|\left(\frac{n}{2}+\frac{2}{3}+\frac{1}{n-1}\right)+\frac{7}{6}+\frac{h}{n-2}-\frac{h}{3}-\frac{n}{6}\\
&<|C'|\left(\frac{n}{2}+\frac{2}{3}+\frac{1}{n-1}\right).
\end{align*}

\smallskip

Now, we may apply \ref{Rule22} to shift share to $\bc$. We shift $h(1/6+1/(n-2)-1/(n-5))/(n-3)$ share to $\bc$. To ease the approximation, we notice that $h(1/6+1/(n-2)-1/(n-5))/(n-3)< 1/6$ since $h\leq n-3$ and we use this when $n\geq13$. In the following, the cases $n\in[11,12]$ are verified by substituting the corresponding value of $n$ to the inequalities $$s(\bc)\leq \frac{n}{2}+\frac{1}{n-1}+\frac{2}{3}+h\frac{1/6+1/(n-2)-1/(n-5)}{n-3}\leq\begin{cases}\frac{n}{2}+\frac{1}{2}+\frac{n-1}{3(n-4)}, &\text{if } 11\leq n\leq 12\\
\frac{n}{2}+\frac{5}{6}+\frac{1}{n-1}, &\text{if } 13\leq n. \end{cases}$$
Moreover, we have $n/2+5/6+1/(n-1)\leq  n/2+1+1/(n^2-5n)-2/(3n)$ when $n\geq 13$. Indeed, when $n\geq13$, we have $1/(n-1)+2/(3n)\leq 1/6$.

Finally, we consider the second case. We shift $(1/6+1/(n-2)-1/(n-5))/(n-3)$ share away from $\bv$ to $n-3$ different codewords using \ref{Rule22}. Observe that \ref{Rule1} does not shift share to $\bv$. Thus, we have

$$s(\bv)\leq \frac{2}{2}+1+\frac{1}{n-2}+\frac{n-3}{2}-(n-3)\frac{\frac{1}{6}+\frac{1}{n-2}-\frac{1}{n-5}}{n-3}=\frac{n}{2}+\frac{1}{3}+\frac{1}{n-5}.$$\end{proof}
\end{lemma}

Finally, we are ready to describe how we shift share to $\bv$. Observe, that for a sparse father $\bu$, we have $I_3(\bu)=I_2(\bu)$ and in the following rule we could as well consider $B_2(\bu)$ instead of $B_3(\bu)$.

\begin{enumerate}\label{sääntö23}
\item[\namedlabel{Rule23}{\textbf{Rule} $3$}:] Let $\bu$ be a sparse father. After applying \ref{Rule1} and \ref{Rule22}, we average out the shares among the codewords in $B_3(\bu)$, that is, the new shares for the codewords in $B_3(\bu)$ become $$\frac{\sum_{\bc\in I_3(\bu)} s(\bc)}{n}.$$
\end{enumerate}

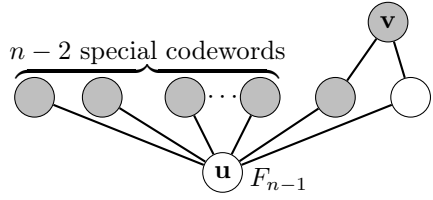
\begin{figure}
\centering
\begin{tikzpicture}


\node[main node](1) at (0,0)    {$\bu$};
\node[code node](2) at (-2.5,1) {};
\node[code node](3) at (-1.6,1) {};
\node[code node](4) at (-0.5,1) {};
\node[main node](5) at (2.5,1)  {};
\node[code node](6) at (1.5,1)  {};
\node[code node](7) at (0.5,1)  {};
\node[code node](8) at (2.2,2)    {$\bv$};

\node[] at (0.03,1) {$\cdots$};
\node[] at (-1.0,1.35) {$\overbrace{\hspace{3.5cm}}$};
\node[] at (-1.0,1.6) {$n-2$ special codewords};
\node[] at (0.75,-0.1) {$F_{n-1}$};

\path[draw,thick]
    (1) edge node {} (2)
    (1) edge node {} (3)
    (1) edge node {} (4)
    (1) edge node {} (5)
    (1) edge node {} (6)
    (1) edge node {} (7)
    (8) edge node {} (5)
    (8) edge node {} (6)

    ;

\end{tikzpicture}\centering
\caption{Gray vertices are codewords. Word $\bu$ is a sparse father. Now, \ref{Rule23} averages the share between the $n-2$ special codewords, the codeword $\bv$ and the codeword in $I(\bv)$. Notice that we apply first \ref{Rule1}, then \ref{Rule22} and after that \ref{Rule23}.}\label{Rule3Esim.}
\end{figure}

In the following theorem, the upper bound is derived by applying first \ref{Rule1} to each suitable codeword, then applying \ref{Rule22} to each suitable codeword and finally applying \ref{Rule23}. 
Applying \ref{Rule23} is illustrated in Figure \ref{Rule3Esim.}. Now we are finally ready to present the proof of Theorem \ref{parasraja}.

\begin{customthm}{7}\label{2/3saavutettu raja}
We have $$\gamma^{LD}(\F^n)\geq \begin{cases}\frac{2^{n+1}}{n+1+2(n-1)/(3(n-4))}, &\text{if } 11\leq n\leq12\\
\frac{2^{n+1}}{n+2+2/(n^2-5n)-4/(3n)},&\text{if }13\leq n.\end{cases}$$
\begin{proof}Let $C\subseteq \F^n$ be a locating-dominating code. Our goal is to show that after applying \ref{Rule1}, \ref{Rule22} and \ref{Rule23} on $C$, we have $$s(\bc)\leq \begin{cases}\frac{n}{2}+\frac{1}{2}+\frac{n-1}{3(n-4)}, &\text{if } 11\leq n\leq12\\
\frac{n}{2}+1+\frac{1}{n^2-5n}-\frac{2}{3n},&\text{if }13\leq n\end{cases}$$ for each codeword $\bc\in C$. By (\ref{shareIdea2}), this gives the claim.  

Let us assume that we have applied \ref{Rule1} on the locating-dominating code $C\subseteq \F^n$.  Observe that when $n\in[11,12]$, we have $n/2+1/2+(n-1)/(3(n-4))>n/2+1-1/(n-1)$ and when $n\geq13$, we have $n/2+1+1/(n^2-5n)-2/(3n)>n/2+1-1/(n-1)$. Thus, if \begin{equation}\label{isoShare}
s(\bc)> 
\begin{cases}
\frac{n}{2}+\frac{1}{2}+\frac{n-1}{3(n-4)}, &\text{if } 11\leq n\leq12\\ 
\frac{n}{2}+1+\frac{1}{n^2-5n}-\frac{2}{3n},&\text{if }13\leq n,
\end{cases}
\end{equation}
%
%
then according to Lemma \ref{TyyppiLemma} (in the cases $n \geq 13$) together with Inequality (\ref{n-3Lasta}) (in the cases $n\in[11,12]$) $\bc$ is a special codeword, and further by Lemma \ref{Ein-1isiäB2ssa} $\bc$ is adjacent to exactly one special father, say $\bu$, which is actually sparse due to Lemma \ref{nKoodisanaaB3:ssa}. 
We will first confirm that after applying \ref{Rule22} and \ref{Rule23} codeword $\bc$ has the desired share. 

Let us assume again, without loss of generality, that $\bu=\nolla$, $I(\bu)=\{\e_i\mid i\in[1,n-1]\}$ and $\bv=\e_2+\e_n\in C$. Since $|I_3(\nolla)|= n$, we have $I(\bv)=\{\e_2,\bv\}$. Hence, $\bv$ and $\e_2$ form a couple. Moreover, since $\bv\in C$, there is a father in $N[\bv]$ and the father has weight three since $|I_3(\nolla)|=n$. Without loss of generality, let $\bv+\e_3$ be this $F_j$-father in $N(\bv)$ for some $j\geq3$. Since $\e_2+\e_3\not\in C$, $\e_3+\e_n\not\in C$ and $\bv+\e_3\not\in C$, we have $j\leq n-2$. Moreover, we notice that there is exactly one sparse father $\bu'$ in $B_3(\bv)$ (namely, $\bu'=\bu$). Indeed, since $\nolla$ is a sparse father, there cannot be another sparse father with weight four or less as $|I_3(\nolla)|=n$. Moreover, since a sparse father does not have any codewords at distance three, we have $d(\bu',\bv)=2$. Consequently, $w(\bu')=0$ and $\bu'=\bu$. Hence, \ref{Rule23} affects the word $\bv$ exactly once. Moreover, notice that by Lemma \ref{EiShareaVlle} \ref{Rule1} or \ref{Rule22} do not shift share into the codeword $\bv$. We will now analyze the share of the codeword $\bv$ in order to approximate how much share the codewords in $I_3(\nolla)$ have after we have applied \ref{Rule23}. In particular, we show that $s(\bv)\leq n/2+1/3+1/(n-5)$ before we have applied \ref{Rule23}.

Let us first assume that we have exactly one $F_{n-2}$-father in $N[\bv]$. Now, we apply \ref{Rule22} to $\bv$ or we have $s(\bv)\leq n/2+1/3+1/(n-5)$. By Lemma \ref{22 siirto}, after applying the rule,  we have $s(\bv)\leq n/2+1/3+1/(n-5)$.

Let us then assume that we do not have exactly one $F_{n-2}$-father in $N[\bv]$. Thus, \ref{Rule22} does not shift share in or out of word $\bv$. In this case, there has to be at least two fathers in $N(\bv)$ since if we have only one father, then there are at least $n-3$ sons and thus, an $F_{n-2}$-father. To maximize the share, we assume that there are exactly two fathers in $N[\bv]$. Indeed, if we have more than three fathers in $N[\bv]$, then $s(\bv)\leq 4/3+2/2+1+(n-6)/2=n/2+1/3$. If we have exactly three fathers, then $$s(\bv)\leq \frac{1}{j_1}+\frac{1}{j_2}+\frac{1}{j_3}+\frac{2}{2}+1+\frac{n-5}{2}\leq\frac{n}{2}+\frac{1}{6}+\frac{1}{n-8}\leq \frac{n}{2}+\frac{1}{3}+\frac{1}{n-5}.$$
Here the second inequality is due to Inequalities (\ref{isaMaara}) and (\ref{Apuyhtälö}). Indeed, Inequality (\ref{isaMaara}) gives that $j_1-1+j_2-1+j_3-1\geq n-5$, then to maximize the share $s(\bv)$ we assume that $j_1+j_2+j_3=n-2$ and then we use Inequality (\ref{Apuyhtälö}) to approximate that $j_1=3$, $j_2=3$ and $j_3=n-8$. The third inequality holds when $n=11$ and if $n\geq12$, then it is strict. Now, we may consider the case with two fathers $$s(\bv)\leq \frac{1}{j_1}+\frac{1}{j_2}+\frac{2}{2}+1+\frac{n-4}{2}\leq\frac{n}{2}+\frac{1}{3}+\frac{1}{n-5}.$$ The latter inequality is again due to Inequalities (\ref{isaMaara}) and (\ref{Apuyhtälö}). Indeed, similarly as in the previous case, Inequality (\ref{isaMaara}) gives that $j_1-1+j_2-1\geq n-4$, then to maximize the share $s(\bv)$ we assume that $j_1+j_2=n-2$ and then we use Inequality (\ref{Apuyhtälö}) to approximate that $j_1=3$ and $j_2=n-5$. Therefore, we have again $s(\bv)\leq n/2+1/3+1/(n-5)$. Now we are ready to apply \ref{Rule23}.

Recall that the $n-1$ codewords in $I(\nolla)$ have share of at most $n/2+1$ after applying \ref{Rule1} by Theorem \ref{perusraja}. Moreover, because $|I_3(\nolla)|=n$, \ref{Rule22} does not shift share into them and \ref{Rule23} affects them at most once. If we now apply \ref{Rule23} to codewords $\bc$ in $B_3(\nolla)$, then we get \begin{equation}\label{sääntö23 tasaus}
n\cdot s(\bc)\leq (n-1)\left(\frac{n}{2}+1\right)+\frac{n}{2}+\frac{1}{3}+\frac{1}{n-5}=n\left(\frac{n}{2}+1+\frac{1}{n^2-5n}-\frac{2}{3n}\right).
\end{equation}
Observe (verify by substituting $n$) that when $n\in [11,12]$, we have $n/2+1+1/(n^2-5n)-2/(3n)<n/2 + 1/2 + (n - 1)/(3 (n - 4))$. Therefore, each codeword $\bc$ considered in Inequality (\ref{isoShare}) has the desired share after applying \ref{Rule1}, \ref{Rule22} and \ref{Rule23}.

Let us finally confirm that if we have for some codeword $\bc\in C$ \begin{equation}\label{raja}s(\bc)\leq \begin{cases}\frac{n}{2}+\frac{1}{2}+\frac{n-1}{3(n-4)}, &\text{if } 11\leq n\leq12\\
\frac{n}{2}+1+\frac{1}{n^2-5n}-\frac{2}{3n},&\text{if }13\leq n\end{cases}
\end{equation} before applying \ref{Rule22} and \ref{Rule23}, then the same bounds also hold after we have applied these rules. First of all, by Lemma \ref{22 siirto}, \ref{Rule22} does not push the share of any codeword over Bound (\ref{raja}). Moreover, \ref{Rule23} cannot increase the share of any codeword over this bound because it averages the shares in $I_3(\bu)$ for some sparse father and if it pushes the share of one codeword over the bound, then the shares of all codewords are over the bound. Moreover, if \ref{Rule23} pushes the shares of codewords over Bound (\ref{raja}), then one of the codewords already had its share above Bound (\ref{raja}) after \ref{Rule1} and \ref{Rule22} have been applied. However, this codeword cannot have received share according to \ref{Rule22} as otherwise a contradiction with Lemma~\ref{22 siirto} follows. Therefore, the share is over Bound~(\ref{raja}) already after \ref{Rule1} has been applied, but now this contradicts with (\ref{sääntö23 tasaus}). Hence, the claim follows.\end{proof}
\end{customthm}

\bibliographystyle{abbrv}\label{Biblio}

\bibliography{LD}
\end{document}